\documentclass[a4paper,11pt,DIV=11,oneside,headings=small,bibliography=totoc,abstract=true,intlimits]{scrartcl}
\usepackage{amssymb,amsmath,amsfonts,amsthm}
\allowdisplaybreaks
\usepackage{mathtools}
\usepackage[english]{babel}
\usepackage[T1]{fontenc}
\usepackage{textgreek}
\usepackage{enumerate}
\usepackage[noblocks]{authblk}

\usepackage[textsize=tiny,shadow]{todonotes}

\usepackage[linktocpage, pdftex,colorlinks]{hyperref}
	\usepackage{color}
 	\definecolor{darkred}{rgb}{0.5,0,0}
 	\definecolor{darkgreen}{rgb}{0,0.5,0}
 	\definecolor{darkblue}{rgb}{0,0,0.5}
 	\hypersetup{colorlinks,linkcolor=darkblue,filecolor=darkgreen,urlcolor=darkred,citecolor=darkblue}
\hypersetup{
pdftitle={Sufficient criteria for stabilization properties in Banach spaces},
pdfauthor={Michela Egidi, Dennis Gallaun, Christian Seifert, Martin Tautenhahn},
}

\newcommand{\thickset}{E}
\newcommand{\ii}{\mathrm{i}}

\newcommand{\drm}{\mathrm{d}}
\newcommand{\euler}{\mathrm{e}}

\newcommand{\N}{\mathbb{N}}
\newcommand{\R}{\mathbb{R}}
\newcommand{\C}{\mathbb{C}}

\newcommand{\cL}{\mathcal{L}}
\newcommand{\F}{\mathcal{F}}

\newcommand{\1}{\mathbf{1}}
\newcommand{\from}{\colon}
\newcommand{\norm}[1]{\left\lVert#1\right\rVert}
\newcommand{\abs}[1]{\left\lvert#1\right\rvert}
\newcommand{\re}{\operatorname{Re}}
\renewcommand{\Re}{\re}

\newcommand{\supp}{\operatorname{supp}}
\newcommand{\id}{\operatorname{Id}}
\newcommand{\ran}{\operatorname{Ran}}
\renewcommand{\epsilon}{\varepsilon}
\newtheorem{theorem}{Theorem}[section]
\newtheorem{proposition}[theorem]{Proposition}
\newtheorem{lemma}[theorem]{Lemma}
\newtheorem{corollary}[theorem]{Corollary}

\theoremstyle{definition}

\newtheorem{definition}[theorem]{Definition}

\theoremstyle{remark}
\newtheorem{remark}[theorem]{Remark}

\title{Sufficient criteria for stabilization properties in Banach spaces}

\author{Michela Egidi}
\affil{Ruhr-Universit\"at Bochum, Fakult\"at f\"ur Mathematik, Universit\"atsstra{\ss}e 150, 44780 Bochum, Germany, michela.egidi@rub.de}

\author{Dennis Gallaun}
\affil{Technische Universit\"at Hamburg, Institut f\"ur Mathematik, Am Schwarzenberg-Campus 3, 21073 Hamburg, Germany, \{dennis.gallaun, christian.seifert\}@tuhh.de}

\author[2,3]{Christian Seifert}
\affil{Christian-Albrechts-Universit\"at zu Kiel, Mathematisches Seminar, Heinrich-Hecht-Platz 6, 24118 Kiel, Germany}

\author{Martin Tautenhahn}
\affil{Universit\"at Leipzig, Mathematisches Institut, Augustusplatz 10, 04109 Leipzig, Germany, martin.tautenhahn@math.uni-leipzig.de}

\date{\vspace{-7ex}}

\begin{document}

\maketitle

\begin{abstract}
We study abstract sufficient criteria for open-loop stabilizability of linear control systems in a Banach space with a bounded control operator, which build up and generalize a sufficient condition for null-controllability in Banach spaces given by an uncertainty principle and a dissipation estimate.  For stabilizability these estimates are only needed for a single spectral parameter and, in particular, their constants do not depend on the growth rate w.r.t. this parameter. Our result unifies and generalizes earlier results obtained in the context of Hilbert spaces.
As an application we consider fractional powers of elliptic differential operators with constant coefficients in $L_p(\R^d)$ for $p\in [1,\infty)$ and thick control sets. 
\\[1ex]
\textsf{\textbf{Mathematics Subject Classification (2010).}} 47D06, 35Q93, 47N70, 93D20, 93B05, 93B07.
\\[1ex]
\textbf{\textsf{Keywords.}} Stabilizability, Banach space, $C_0$-semigroups, observability estimate, null-con\-trollabi\-li\-ty, fractional powers
\end{abstract}
\section{Introduction}
Let $X,U$ be Banach spaces, $(S_t)_{t\geq 0}$ a $C_0$-semigroup on $X$ with generator $-A$, $B\in \cL(U,X)$, $x_0\in X$. We consider the control system
\begin{equation}
\dot{x}(t) = -Ax(t) + Bu(t), \qquad t> 0, \qquad x(0) = x_0
\label{eq:IntroControlSystem}
\end{equation}
with a control function $u\in L_r((0,\infty);U)$ for some $r\in [1,\infty]$. In this paper we focus on the question whether the system \eqref{eq:IntroControlSystem} is open-loop stabilizable; that is, there is a control function $u\in L_r((0,\infty);U)$ such that the corresponding mild solution decays exponentially. We give a sufficient condition for open-loop stabilizability which is based on a well-known strategy to prove null-controllability. The system \eqref{eq:IntroControlSystem} is called null-controllable in time $T>0$ if there is a control function $u\in L_r((0,T);U)$ such that the corresponding solution of \eqref{eq:IntroControlSystem} satisfies $x(T)=0$. Clearly, null-controllability implies stabilizability. We weaken sufficient conditions for null-controllability to obtain more general criteria for stabilizability.

One possible approach to prove null-controllability is a method known as Lebeau-Robbiano strategy, originating in the seminal work by Lebeau and Robbiano \cite{LebeauR-95}, see also \cite{LebeauZ-98,JerisonL-99}. Subsequently, this strategy was generalised in various steps to $C_0$-semigroups on Hilbert spaces, see, e.g., \cite{Miller-10,TenenbaumT-11,WangZ-17,BeauchardP-18,NakicTTV-20}, and more recently to $C_0$-semigroups on Banach spaces, see \cite{GallaunST-20,BombachGST-20}. 
The essence of this approach is to show an uncertainty principle and a dissipation estimate for the dual system which are valid for an infinite sequence of so-called spectral parameters, and prove that the growth rate in the uncertainty principle is strictly smaller than the decay rate of the dissipation estimate. 
In Section \ref{sec:sufficient_condition} we show that for proving stabilizability in general Banach spaces one can drop the assumption on the growth and decay rate in the estimates. This was first observed in \cite{HuangWW-20,LiuWXY-20} in the context of Hilbert spaces.
Similar to what was used in a proof in \cite{LiuWXY-20}, we show that it is sufficient to prove the uncertainty principle and the dissipation estimate only for one single spectral parameter. This leads to a plain condition for stabilizability in Banach spaces which does not involve assumptions on the constant in the uncertainty principle. Let us stress that the latter improvement allows to apply our result to models where an uncertainty principle is avaible only for some spectral parameters as in \cite{LenzSS-20}. We will pursue this application in a forthcoming paper.

In order to prove the sufficient condition for stabilizability we introduce in Section \ref{sec:concepts} two auxiliary concepts, namely $\alpha$-controllability and a weak observability inequality. Similar to a result in \cite{TrelatWX-20} for Hilbert spaces, we show a duality result for these concepts in general Banach spaces. In order to deal with this more general framework, we use a separation theorem instead of a Fenchel-Rockafellar duality argument applied in \cite{TrelatWX-20}.

Finally, in Section \ref{sec:application}, we verify the sufficient conditions for fractional powers of elliptic differential operators $-A$ with constant coefficients on $L_p(\R^d)$ for $p\in [1,\infty)$ and where $B=\1_\thickset \from L_p (\thickset)\to L_p(\R^d)$ is the embedding from a so-called thick set $\thickset \subset \R^d$ to $\R^d$. This complements recent results in the Hilbert space $L_2(\R^d)$ for the fractional heat equation and more general Fourier multipliers, see \cite{HuangWW-20,Lissy-20,LiuWXY-20,Koenig-20,AlphonseM-21}.
 
\section{Stabilizability and related concepts}\label{sec:concepts}

Let $X,U$ be Banach spaces, $(S_t)_{t\geq 0}$ a $C_0$-semigroup on $X$ with generator $-A$, $B\in \cL(U,X)$, and $x_0\in X$. 
We consider the control system
\begin{equation}
\dot{x}(t) = -Ax(t) + Bu(t), \qquad t>0, \qquad x(0) = x_0
\label{eq:ControlSystem}
\end{equation}
where $u\in L_r((0,\infty);U)$ with some $r\in [1,\infty]$.
The unique mild solution of \eqref{eq:ControlSystem} is given by Duhamel's formula
$$x(t)  = S_tx_0 + \int_0^t S_{t-\tau}Bu(\tau) \,\drm \tau, \qquad t>0.$$
For $t > 0$ the \emph{controllability map} $L_t \in \cL(L_r((0,t);U),X)$ is given by
\begin{equation}
L_t u = \int_0^t S_{t-\tau}Bu(\tau) \,\drm \tau.
\label{eq:controllability_map}
\end{equation}

\begin{definition}
    The system \eqref{eq:ControlSystem} is called \emph{open-loop stabilizable w.r.t.\ $L_r((0,\infty);U)$} if there are $M\geq 1$ and $\omega <0$ such that for all $x_0\in X$ there exists $u \in L_r((0,\infty);U)$ such that
\begin{equation}
\lVert x(t) \rVert = \lVert S_t x_0 + L_tu \rVert \le M \euler^{\omega t} \lVert x_0 \rVert, \quad t\geq 0.
\label{eq:stabilizability}
\end{equation} 
Moreover, we call \eqref{eq:ControlSystem} \emph{cost-uniformly open-loop stabilizable w.r.t.\ $L_r((0,\infty);U)$} if there exists $M\geq 1$, $\omega <0$, and $C\geq 0$ such that for all $x_0\in X$ there exists $u \in L_r((0,\infty);U)$ such that 
	$$ \lVert u \rVert_{L_r((0,\infty);U)} \le C \lVert x_0\rVert \quad \text{and}\quad \lVert x(t) \rVert = \lVert S_t x_0 + L_tu \rVert \le M \euler^{\omega t} \lVert x_0 \rVert, \quad t\geq 0.$$
\end{definition}
\begin{remark}
Sometimes \eqref{eq:stabilizability} is replaced by the weaker condition $x\in L_2((0,\infty),X)$. For $r=2$ this is also called \emph{optimizability} or \emph{finite cost condition}.  Recall that one says that the system \eqref{eq:ControlSystem} is \emph{closed-loop stabilizable} or \emph{stabilizable by feedback} if there exists $K\in \cL(X,U)$ such that $-A+BK$ generates an exponentially stable $C_0$-semigroup. Then $K$ is called \emph{state feedback operator} and the control $u$ given by $u(t)=Kx(t)$ yields an exponentially stable solution $x$.
For an open-loop stabilizable system in a Hilbert space, the existence of a state feedback operator follows from classical Riccati theory, see e.g. \cite[Theorem IV.4.4]{Zabczyk-08}. Hence in Hilbert spaces every open-loop stabilizable system is also cost-uniformly open-loop stabilizable.
\end{remark}
Next we introduce two concepts, namely $\alpha$-controllability and weak observability inequalities, and discuss their close connection to open-loop stabilizability.

\subsection{\texorpdfstring{$\alpha$}{\textalpha}-controllability}

In this section we define $\alpha$-controllability and show that for $\alpha\in [0,1)$ it is equivalent to cost-uniform open-loop stabilizability.

\begin{definition}
    Let $\alpha\geq 0$. The system \eqref{eq:ControlSystem} is called \emph{$\alpha$-controllable in time $T$ w.r.t.\ $L_r((0,T);U)$} if for all $x_0\in X$ there exists $u \in L_r((0,T);U)$ such that 
	$$\lVert x(T) \rVert = \lVert S_T x_0 + L_Tu \rVert \le \alpha \lVert x_0 \rVert.$$
	Moreover, we call \eqref{eq:ControlSystem} \emph{cost-uniformly $\alpha$-controllable in time $T$ w.r.t.\ $L_r((0,T);U)$} if there exists $C\geq 0$ such that for all $x_0\in X$ there exists $u \in L_r((0,T);U)$ such that 
	$$ \lVert u \rVert_{L_r((0,T);U)} \le C \lVert x_0 \rVert \quad \text{and}\quad \lVert x(T) \rVert = \lVert S_T x_0 + L_Tu \rVert \le \alpha \lVert x_0 \rVert.$$
\end{definition}

\begin{remark}
For $\alpha=0$ the concept of $0$-controllability coincides with the usual notion of \emph{null-controllability}. If the system \eqref{eq:ControlSystem} is $\alpha$-controllable for all $\alpha >0$, it is usually called \emph{approximate null-controllable}. For the control system \eqref{eq:ControlSystem}, the quantity $\norm{u}_{L_r((0,T);U)}$ is called \emph{cost}. An $\alpha$-controllable system is in general not cost-uniformly $\alpha$-controllable, see \cite[Section 3.2.1]{TrelatWX-20}. However, if $\alpha=0$ these two notions are equivalent, see \cite[Theorem 2.2]{Carja-89}. 
\end{remark}

Similarly to \cite[Lemma 31]{TrelatWX-20} (see also \cite[Theorem 26]{TrelatWX-20}) we obtain the following relationship between cost-uniform $\alpha$-controllability and cost-uniform open-loop stabilizability.

\begin{proposition} \label{prop:solution-stable}
The system \eqref{eq:ControlSystem} is cost-uniformly open-loop stabilizable if and only if there exists $\alpha\in [0,1)$ and $T>0$ such that \eqref{eq:ControlSystem} is cost-uniformly $\alpha$-controllable in time $T$. 
\end{proposition}

\begin{proof}
Assume that \eqref{eq:ControlSystem} is cost-uniformly open-loop stabilizable, i.e.\ for all $x_0\in X$ there exists $u\in L_r((0,\infty);U)$ such that the solution of \eqref{eq:ControlSystem} satisfies $\lVert x(t) \rVert = \lVert S_t x_0 + L_tu \rVert \le M \euler^{\omega t} \lVert x_0 \rVert$ for all $t>0$ with uniform parameters $M\geq 1$ and $\omega<0$. For all $\alpha\in (0,1)$ there exists $T>0$ such that $M \euler^{\omega T} \le \alpha$ and hence \eqref{eq:ControlSystem} is $\alpha$-controllable in time $T$. Moreover, since the cost $\norm{u}_{L_r((0,\infty);U)}$ can be controlled uniformly w.r.t.\ the initial value $x_0$, the system \eqref{eq:ControlSystem} is even cost-uniformly $\alpha$-controllable in time $T$.

We now show the converse and assume that \eqref{eq:ControlSystem} is cost-uniformly $\alpha$-controllable in time $T$. For $\alpha=0$ we have $x(T) = 0$ and therefore $x(t)=0$ for all $t\geq T$, so the statement is trivial.
Thus, let $\alpha\in (0,1)$. Let $x_0 \in X$ and $u_0 \in L_r((0,T);U)$ such that $\lVert u_0 \rVert_{L_r((0,T);U)} \le C \lVert x_0 \rVert$ and $\lVert S_Tx_0 + L_T u_0 \rVert \le \alpha \lVert x_0 \rVert$. For $k\in \N_0$ we recursively define $x_{k+1} = S_Tx_k + L_T u_k$ and choose $u_k \in L_r((0,T);U)$ such that
$$ \lVert u_k \rVert_{L_r((0,T);U)} \le C \lVert x_k \rVert \quad \text{and} \quad \lVert S_Tx_k + L_T u_k \rVert \le \alpha \lVert x_k \rVert.$$
Define $u\from [0,\infty) \to U$ as the concatenation
$$u(t) = u_k(t-kT) \quad \text{if } t\in [kT,(k+1)T).$$
Then, $\norm{x_k} \leq \alpha^k \norm{x_0}$ for all $k\in\N_0$. For $r\in [1,\infty)$, we have
\begin{align*}
\lVert u \rVert^r_{L_r((0,\infty);U)} &= \int_0^\infty \lVert u(\tau) \rVert^r \drm \tau \le \sum_{k=0}^\infty \int_{kT}^{(k+1)T} \lVert u(\tau) \rVert^r \drm \tau \le C^r \sum_{k=0}^\infty  \lVert x_{k+1} \rVert^r \\
&\le C^r \sum_{k=0}^\infty \alpha^{rk} \lVert x_{0} \rVert^r \le C^r \frac{1}{1-\alpha^r} \lVert x_0 \rVert^r,
\end{align*}
and hence $u\in L_r((0,\infty);U)$. For $r=\infty$, we similarly estimate
\[ \lVert u \rVert_{L_\infty((0,\infty);U)} = \sup_{k\in\N_0} \norm{u_k}_{L_\infty((0,T);U)}\leq C\sup_{k\in\N_0} \norm{x_k}\leq C\sup_{k\in\N_0} \alpha^k \norm{x_0} \leq C\norm{x_0},\]
and therefore also $u\in L_\infty((0,\infty);U)$.

The control $u$ generates a trajectory 
$$x(t) = S_tx_0 + \int_0^t S_{t-\tau} B u(\tau) \drm \tau, \quad t>0$$
satisfying $x(kT) = x_{k}$ for all $k\in\N_0$.
Let $M_S\geq 1$ such that $\sup_{t\in [0,T]} \lVert S_t \rVert_{\cL(X)} \le M_S$. Then for all $k\in\N_0$ and $t\in [kT,(k+1)T)$, by H\"older's inequality, we have
\begin{align*}
\lVert x(t) \rVert &= \Bigl\lVert S_{t-kT}x_k + \int_0^{t-kT} S_{t-kT-\tau} B u_k(\tau-kT) \drm \tau \Bigr\rVert
\le M_S \lVert x_k \rVert + M_S  \lVert B \rVert \int_0^T \lVert u_k(\tau) \rVert \drm \tau \\
&\le M_S \lVert x_k \rVert + M_S \lVert B \rVert  T^{1/{r'}} \lVert u_k \rVert_{L_r((0,T);U)}
\le M_S (1+ \lVert B \rVert T^{1/{r'}} C)\alpha^k \lVert x_0 \rVert,
\end{align*}  
where $r'\in [1,\infty]$ such that $1/r + 1/r' =1$ (and $1/\infty = 0$ as usual). 
Since $\ln\alpha < 0$ and $\alpha^{k+1} = \euler^{(k+1)T \frac{\ln \alpha}{T}} \le \euler^{\frac{\ln \alpha}{T} t}$ for $t\in [kT, (k+1)T)$ we infer that
$$ \lVert x(t) \rVert \le \frac{M_S}{\alpha}(1+ \lVert B \rVert T^{1/{r'}} C) \euler^{\frac{\ln \alpha}{T} t} \lVert x_0 \rVert.$$
Thus, we obtain the assertion with $M= \frac{M_S}{\alpha}(1+ \lVert B \rVert T^{1/{r'}} C)\geq 1$ and $\omega = \ln \alpha /T <0$.
\end{proof}
\subsection{Weak observability inequalities}

In this section we prove the duality between cost-uniform $\alpha$-controllability and a weak observability estimate for the dual system. 

\begin{definition}
    Let $X,Y$ be Banach spaces, $(S_t)_{t\geq 0}$ a semigroup on $X$, $C\in\cL(X,Y)$, $T>0$, and assume that $[0,T]\ni t\mapsto \norm{CS_t x}_Y$ is measurable for all $x\in X$. Let $r\in[1,\infty]$. Then we say that a \emph{weak observability inequality} is satisfied if there exist $C_{\mathrm{obs}}\geq 0$ and $\alpha\geq 0$ such that for all $x\in X$ we have
\begin{equation}\label{eq:weak_obs}
\norm{S_Tx}_X\leq \begin{cases}
  C_{\mathrm{obs}} \left( \int_0^T \lVert C S_t x \rVert_{Y}^{r} \drm t \right)^{1/r} + \alpha \lVert x \rVert_{X}& \text{if } r \in [1,\infty) , \\
  C_{\mathrm{obs}} \sup\limits_{t \in [0,T]} \lVert C S_t x \rVert_{Y} + \alpha \lVert x \rVert_{X}& \text{if } r = \infty.
  \end{cases}
\end{equation} 
\end{definition}

\begin{remark}
For $\alpha=0$ the weak observability inequality coincides with the usual \emph{observability inequality} which corresponds to so-called \emph{final state observability}.
Note that for all $C_0$-semigroups with $\lVert S_t \rVert \le M \euler^{\omega t}$ for $t\geq 0$ inequality \eqref{eq:weak_obs} holds with $\alpha = M\euler^{\max\{\omega,0\} T}$ for all $C_{\mathrm{obs}}\geq 0$ and all operators $C\in\cL(X,Y)$. However, we are mainly interested in the case $\alpha\in [0,1)$, where weak observability inequalities are linked to open-loop stabilizability of the predual system, see Proposition \ref{prop:solution-stable} and the following Theorem \ref{thm:duality}.
\end{remark}

\begin{theorem}\label{thm:duality}
Let $X,U$ be Banach spaces, $(S_t)_{t\geq 0}$ a $C_0$-semigroup on $X$, $T>0$, $r\in [1,\infty]$ and $L_T\in \cL(L_r((0,T);U),X)$ the controllability map defined in \eqref{eq:controllability_map}. Let further $C\geq 0$ and $\alpha \geq 0$. Then the following statements are equivalent:
\begin{enumerate}[(a)]
	\item \label{thm:duality_s1} For every $x \in X$ and $\epsilon >0$ there exists $u\in L_r((0,T);U)$ with
	\begin{equation*}
	\lVert u \rVert_{L_r((0,T);U)} \le C  \lVert x \rVert_X \quad  \text{and} \quad \lVert S_Tx+L_Tu \rVert_X < (\alpha+\epsilon) \lVert x \rVert_X.
	\end{equation*}
		\item \label{thm:duality_s2} For all $x' \in X'$ we have
		\begin{equation*}
		\lVert S_T'x' \rVert_{X'} \le \begin{cases}
  C \left( \int_0^T \lVert B' S_t' x' \rVert_{U'}^{r'} \drm t \right)^{1/r'} + \alpha \lVert x' \rVert_{X'}& \text{if } r' \in [1,\infty) , \\
  C \sup\limits_{t \in [0,T]} \lVert B' S_t' x' \rVert_{U'} + \alpha \lVert x' \rVert_{X'}& \text{if } r' = \infty,
  \end{cases}
		\end{equation*}
	where $r'\in [1,\infty]$ with $1/r + 1/r' = 1$.
\end{enumerate}
\end{theorem}

\begin{remark}
 Theorem \ref{thm:duality} can be rephrased as: cost-uniform $\alpha$-controllability for \eqref{eq:ControlSystem} is equivalent to a weak observability inequality of the corresponding dual system.
Note that in the case $\alpha=0$ the above theorem gives the well-known duality between approximate null-controllability and final state observability.
\end{remark}

In contrast to \cite{TrelatWX-20} we do not use a Fenchel-Rockafellar duality argument to prove Theorem \ref{thm:duality}, but the following well-known separation theorem. We cite here a version from \cite[Lemma 1.2]{Carja-89}, for a proof see \cite[Theorem I.5.10, Lemma II.4.1]{Goldberg-66}. 
\begin{lemma}\label{lem:separation}
Let $A,B$ be convex sets in a Banach space $X$. Then $A \subset \overline{B}$ if and only if 
$$\sup_{x\in A} \langle x,x'\rangle_{X,X'} \le \sup_{x\in B} \langle x,x'\rangle_{X,X'} \qquad \text{for all } x'\in X'.$$
\end{lemma}

\begin{proof}[Proof of Theorem \ref{thm:duality}]
We consider the convex sets
$$A = \{S_Tx : \lVert x \rVert_X \le 1\} \quad \text{and}\quad B = \{L_Tu + \alpha x : \lVert u \rVert_{L_r((0,T);U)} \le C, \lVert x \rVert_X \le 1\}.$$
We observe that the following three statements are equivalent:
\begin{enumerate}[(a)]
 \item \label{s1}$A\subset \overline{B}$
 \item \label{s2} for all $\epsilon >0$ and $x_1\in X$ with $ \lVert x_1 \rVert_X\le 1$ there exists $u\in L_r((0,T);U)$ with $\lVert u \rVert_{L_r((0,T);U)} \allowbreak \le C$  and $x_2\in X$ with $\lVert x_2 \rVert_X \le 1$ such that 
$$\lVert S_Tx_1+L_Tu+\alpha x_2 \rVert_X < \epsilon.$$ 
\item \label{s3} for all $\epsilon >0$ and $x_1\in X$ with $\lVert x_1 \rVert_X\le 1$ there exists $u\in L_r((0,T);U)$ with $\lVert u \rVert_{L_r((0,T);U)} \allowbreak \le C$ such that $$\lVert S_Tx_1+L_Tu \rVert_X < \alpha+\epsilon .$$
\end{enumerate}
While \eqref{s1}$\Leftrightarrow$\eqref{s2} and \eqref{s2}$\Rightarrow$\eqref{s3} are obvious, we note that \eqref{s2} follows from \eqref{s3} by choosing $x_2 = -  (S_Tx_1+L_Tu) / (\alpha + \epsilon)$. Since 
\[
\bigl\lVert S_Tx / \lVert x \rVert +L_T u \bigr\rVert_X = \frac{1}{\lVert x \rVert}   \\\bigl\lVert S_T x+L_T \lVert x \rVert u \bigr\rVert_X 
\]
for all $x\in X\setminus \{0\}$, we find that \eqref{s3} (and thus also \eqref{s1} and \eqref{s2}) is equivalent to statement \eqref{thm:duality_s1} of the theorem. 
Next, for $x'\in X'$ we compute
\begin{equation*}
\sup_{x\in A} \langle x,x'\rangle_{X,X'} = \sup_{\lVert x \rVert_X\le 1} \langle S_Tx,x'\rangle_{X,X'} = \lVert S_T'x' \rVert_{X'}
\end{equation*}
and 
\begin{align*}
\sup_{x\in B} \langle x,x'\rangle_{X,X'} 
& = \sup_{\substack{\lVert u \rVert_{L_r((0,T);U)}\le C,\\ \lVert x \rVert_X\le 1}} \langle L_Tu+\alpha x,x'\rangle_{X,X'}\\
& = \sup_{\lVert u \rVert_{L_r((0,T);U)}\le C} \langle L_Tu,x'\rangle_{X,X'} + \sup_{\lVert x \rVert_X\le 1} \alpha \langle x,x'\rangle_{X,X'} \\
&= C\lVert L_T'x' \rVert_{(L_r((0,T);U))'} + \alpha \lVert x' \rVert_{X'}.
\end{align*}
Finally by \cite[Theorem 2.1]{Vieru-05} we have
$$\lVert L_T'x' \rVert_{(L_r((0,T);U))'} = \begin{cases}
  \left( \int_0^T \lVert B' S_t' x' \rVert_{U'}^{r'} \drm t \right)^{1/r'} & \text{if } r' \in [1,\infty) , \\
  \sup\limits_{t \in [0,T]} \lVert B' S_t' x' \rVert_{U'} & \text{if } r' = \infty,
  \end{cases}$$
	where $r'\in [1,\infty]$ such that $1/r + 1/r' = 1$. Hence $\sup_{x\in A} \langle x,x'\rangle_{X,X'} \le \sup_{x\in B} \langle x,x'\rangle_{X,X'}$ is equivalent to statement \eqref{thm:duality_s2} of the theorem and the claim follows from Lemma \ref{lem:separation}.
\end{proof}

\section{Sufficient conditions for stabilizability}\label{sec:sufficient_condition}

In this section we give a sufficient condition for weak observability inequalities in terms of an uncertainty principle and a dissipation estimate, similar to \cite{HuangWW-20,LiuWXY-20}. 
We emphasize that instead of assuming the uncertainty principle and the dissipation estimate for a family $(P_{\lambda})_{\lambda > 0}$ with certain dependencies of the constants on the ``spectral parameter'' $\lambda$, we need these assumptions to hold only for one single operator $P$. We will relate our result to Lemma 2.2 in \cite{HuangWW-20} and Theorem 2.1 in \cite{GallaunST-20}. Using duality we give, similar to \cite[Theorem 4.1]{LiuWXY-20}, a sufficient condition for open-loop stabilizability in Banach spaces without any compatible condition between the uncertainty principle and a dissipation estimate.

\begin{proposition}\label{lem:spectral+diss-obs}
Let $X$ and $Y$ be Banach spaces, $C\in \cL(X,Y)$, $P\in \cL(X)$, $(S_t)_{t\geq 0}$ a semigroup on $X$, $M \geq 1$ and $\omega \in \R$ such that $\lVert S_t \rVert \leq M \euler^{\omega t}$ for all $t \geq 0$, and assume that for all $x\in X$ the mapping $t\mapsto \lVert C S_t x \rVert_Y$ is measurable.
Further, let $r \in [1,\infty]$, $T>0$ and $C_1,C_2\from(0,T]\to [0,\infty)$ continuous functions such that for all $x\in X$ and $t\in (0,T]$ we have
\begin{align} 
\lVert P S_t x \rVert_{ X } \le C_1(t) \lVert C  P S_t x \rVert_{Y }
\label{eq:ass:uncertainty} , 
\end{align}
and
\begin{align}
\lVert (\id-P) S_t x \rVert_{X} \le C_2(t) \lVert x \rVert_{X} \label{eq:ass:dissipation}.
\end{align}
Then there exist $C_{\mathrm{obs}}\geq 0$ and $\alpha\geq 0$ with
\begin{equation} \label{eq:obs}
\forall x\in X: \quad \norm{S_Tx}_X\leq \begin{cases}
  C_{\mathrm{obs}} \left( \int_0^T \lVert C S_t x \rVert_{Y}^{r} \drm t \right)^{1/r} + \alpha \lVert x \rVert_{X}& \text{if } r \in [1,\infty) , \\
  C_{\mathrm{obs}} \sup\limits_{t \in [0,T]} \lVert C S_t x \rVert_{Y} + \alpha \lVert x \rVert_{X}& \text{if } r = \infty
  \end{cases}
\end{equation}
satisfying for all $\delta\in [0,1)$
\begin{align*} 
C_{\mathrm{obs}} \le \frac{M\euler^{\omega_+ T}}{(1-\delta)T^{1/r}} \max_{t\in [\delta T,T]} C_1(t) \quad \text{and} \quad
\alpha \le \frac{M\euler^{\omega_+ T}}{(1-\delta)T} \int_{\delta T}^T \left(C_1(t)\lVert C \rVert_{\cL (X,Y)}+1\right) C_2(t) \drm t,
\end{align*}
where $\omega_+ = \max\{\omega,0\}$ and $T^{1/r} = 1$ if $r=\infty$.
\end{proposition}
\begin{proof}
Assume we have shown the statement of the proposition in the case $r=1$, i.e.\ for all $x \in X$ we have
 \[
 \lVert S_T x \rVert_{X} \leq C_{\mathrm{obs}} \int_0^T \lVert CS_tx \rVert_Y \drm t +\alpha \lVert x \rVert_X.
 \]
 Then, for all $r \in [1,\infty]$ and all $x \in X$ using H\"older's inequality we obtain  
 \[
  \lVert S_T x \rVert_{X} \leq C_{\mathrm{obs}} T^{1/r'} \left(\int_0^T \lVert CS_tx \rVert_Y^r \drm t \right)^{1/r} + \alpha \lVert x \rVert_X,
 \]
 where $r' \in [1,\infty]$ is such  that $1/r + 1/r' = 1$. Since $T^{-1}T^{1/r'} = T^{-1/r}$ the statement of the proposition follows. Thus, it is sufficient to prove the case $r = 1$.
 \par
Let $t \in (0,T]$ and $x\in X$. Using \eqref{eq:ass:uncertainty} and \eqref{eq:ass:dissipation} we obtain
\begin{align}
 \lVert S_tx \rVert & \leq \lVert PS_tx \rVert + \lVert (\id-P)S_tx \rVert
\leq  C_1(t) \lVert CPS_tx \rVert + \lVert (\id-P)S_tx \rVert \nonumber \\
& \leq  C_1(t) \lVert CS_tx \rVert + C_1(t) \lVert C \rVert_{\cL (X,Y)} \lVert(\id-P)S_tx \rVert  + \lVert(\id-P)S_tx \rVert  \nonumber \\
& \leq  C_1(t) \lVert CS_tx \rVert  + \left(C_1(t)\lVert C \rVert_{\cL (X,Y)}+1\right) C_2(t) \lVert x \rVert_X. \label{eq:to_be_integrated}
\end{align}
Since $(S_t)_{t\geq 0}$ is a semigroup we get
$$\lVert S_Tx \rVert = \lVert S_{T-t} S_t x\lVert  \leq M \euler^{\omega_+ T} \lVert S_tx \rVert,$$
where $\omega_+ = \max\{\omega,0\}$. Since $t\mapsto \lVert C S_t x \rVert_Y$ is measurable by assumption, integrating \eqref{eq:to_be_integrated} with respect to $t\in [\delta T,T]$ we obtain
\begin{align*}
\frac{(1-\delta)T}{M\euler^{\omega_+T}}\lVert S_Tx \rVert & \leq \int_{\delta T}^T C_1(t)\lVert CS_tx \rVert \drm t + \int_{\delta T}^T \left(C_1(t)\lVert C \rVert_{\cL (X,Y)}+1\right) C_2(t) \drm t \,\lVert x \rVert_X \\
& \leq \max_{t\in [\delta T,T]} C_1(t) \int_{\delta T}^T \lVert CS_tx \rVert \drm t + \int_{\delta T}^T \left(C_1(t)\lVert C \rVert_{\cL (X,Y)}+1\right) C_2(t) \drm t \,\lVert x \rVert_X.
\end{align*}
The claim now follows by estimating $\int_{\delta T}^T \lVert CS_tx \rVert \drm t \le \int_{0}^T \lVert CS_tx \rVert \drm t$ and multiplying both sides by $M\euler^{\omega_+T}/(1-\delta)T$.
\end{proof}

The advantage of Proposition \ref{lem:spectral+diss-obs} is the explicit dependence of $C_{\mathrm{obs}}$ and $\alpha$ on the functions $C_1, C_2$ which allows to give conditions to ensure $\alpha \in [0,1)$. By Theorem \ref{thm:duality} and Proposition \ref{prop:solution-stable}, the case where $\alpha \in [0,1)$ is important to prove open-loop stabilizability for the predual system.

\begin{remark}
In Proposition \ref{lem:spectral+diss-obs} we can replace the uncertainty principle in \eqref{eq:ass:uncertainty} by
\begin{equation*}
\forall x\in X: \quad \lVert PS_{T_0}x \rVert \le \begin{cases}
  C_1 \left( \int_0^{T_0} \lVert C P S_t x \rVert_{Y}^{r} \drm t \right)^{1/r} & \text{if } r \in [1,\infty) , \\
  C_1 \sup\limits_{t \in [0,T]} \lVert C P S_t x \rVert_{Y} & \text{if } r = \infty
  \end{cases}
\end{equation*}	
for some $C_1 >0$ and $0<T_0 \le T$. Similar as in the proof of Proposition \ref{lem:spectral+diss-obs}, for $r \in [1,\infty)$ we then estimate 
	\begin{align*}
		& \norm{S_{T_0} x}_{X} \leq \norm{PS_{T_0} x}_{X}+\norm{(\id -P) S_{T_0} x}_{X}
		\leq C_1 \Big( \int_0^{T_0} \norm{C P S_t x}_{Y}^r \drm t\Big)^{1/r} + C_2(T_0)\norm{x}_{X}\\
		& \leq C_1 \Big( \int_0^{T_0} 2^{r-1}( \norm{C S_t x}_{Y}^r +\norm{C}_{\mathcal{L}(X,Y)}^r C_2(t)^r \norm{x}_{X}^r) \drm t \Big)^{1/r} +C_2(T_0)\norm{x}_{X}\\
		& \leq 2^{1-1/r} C_1 \Big( \int_0^{T_0} \norm{C S_t x}_{Y}^r \drm t \Big)^{1/r} +\Big(2^{1-1/r} C_1\norm{C}_{\mathcal{L}(X,Y)} \norm{C_2}_{L_r(0,T_0)}  + C_2(T_0)\Big)\norm{x}_{X}.
	\end{align*}
	Since $\norm{S_{T} x}_{X}=\norm{S_{T- T_0} S_{T_0} x}_{X} \leq M e^{w_+ T}\norm{S_{T_0}x}_{X}$, we obtain \eqref{eq:obs} with
\begin{align*} 
C_{\mathrm{obs}} \le M\euler^{\omega_+ T} 2^{1-1/r} C_1 \quad \text{and} \quad
\alpha \le M\euler^{\omega_+ T} \Big(2^{1-1/r} C_1\norm{C}_{\mathcal{L}(X,Y)} \norm{C_2}_{L_r(0,T_0)}  + C_2(T_0)\Big).
\end{align*}
The case $r=\infty$	is similar and the term $2^{1-1/r}$ can be set to $1$. 
\end{remark}

\begin{remark}
\label{rem:application}
Let us relate Proposition \ref{lem:spectral+diss-obs} to the results obtained in \cite{HuangWW-20} and \cite{GallaunST-20,BombachGST-20}. By choosing the functions $C_1,C_2\from (0,T] \to [0,\infty)$ appropriately we can mimic the assumptions of \cite[Lemma 2.2]{HuangWW-20} and \cite[Theorem 2.1]{GallaunST-20}, respectively. For given $T,\lambda>0$ suppose we have for all $x\in X$ and $t\in (0,T]$ the inequalities \eqref{eq:ass:uncertainty} and \eqref{eq:ass:dissipation} with
\begin{equation}
C_1(t) = d_0 \euler^{d_1\lambda^{\gamma_1}} \quad \text{and} \quad C_2(t) = d_2 \euler^{-d_3\lambda^{\gamma_2}t^{\gamma_3}},
\label{eq:choiceC1C2}
\end{equation} 
where $d_0,d_1,d_2,d_3,\gamma_1,\gamma_2,\gamma_3>0$. Then Proposition \ref{lem:spectral+diss-obs} implies for all $\delta \in (0,1)$ the weak observability inequality \eqref{eq:obs} with 
\begin{align*} 
C_{\mathrm{obs}} \le \frac{Md_0}{\delta T^{1/r}} d_0 \euler^{d_1\lambda^{\gamma_1}+\omega_+ T} \quad \text{and} \quad
\alpha \le Md_2\left(d_0\lVert C \rVert+1\right) \euler^{-d_3\lambda^{\gamma_2}(\delta T)^{\gamma_3}+d_1\lambda^{\gamma_1}+\omega_+ T}.
\end{align*}
Imposing conditions on $T$ and $\lambda$ we can achieve $\alpha\in [0,1)$. We list here only some interesting cases:
\begin{enumerate}[(a)]
	\item\label{case1} Assume $\gamma_1>\gamma_2$. Let $\gamma_3>1-\gamma_2/\gamma_1$, i.e.\ $\gamma_1\gamma_3/(\gamma_1-\gamma_2)>1$, and $T>0$ large enough such that
	$$\ln \left(Md_2(d_0\lVert C \rVert+1)\right) < \left(\frac{d_3}{2 d_1}\right)^{\frac{\gamma_2}{\gamma_1-\gamma_2}} \frac{d_3}{2} (\delta T)^{\frac{\gamma_1\gamma_3}{\gamma_1-\gamma_2}} - \omega_+ T.$$
	Then for $\lambda = \left(\frac{d_3 (\delta T)^{\gamma_3}}{2 d_1}\right)^{\frac{1}{\gamma_1-\gamma_2}}$ we have $\alpha\in (0,1)$. Indeed, one easily computes
	\begin{align*} 
\alpha &\le Md_2\left(d_0\lVert C \rVert+1\right) \exp\Bigl(-d_3\bigl(\tfrac{d_3 (\delta T)^{\gamma_3}}{2 d_1}\bigr)^{\frac{\gamma_2}{\gamma_1-\gamma_2}} (\delta T)^{\gamma_3}+d_1\bigl(\tfrac{d_3 (\delta T)^{\gamma_3}}{2 d_1}\bigr)^{\frac{\gamma_1}{\gamma_1-\gamma_2}}+\omega_+ T\Bigr) \\
&= Md_2\left(d_0\lVert C \rVert+1\right) \exp \Bigl(- \bigl(\tfrac{d_3}{2 d_1}\bigr)^{\frac{\gamma_2}{\gamma_1-\gamma_2}} \tfrac{d_3}{2} (\delta T)^{\frac{\gamma_1\gamma_3}{\gamma_1-\gamma_2}}+\omega_+ T\Bigr) < 1.
\end{align*}
	\item\label{case2} Assume $\gamma_1 = \gamma_2$. Let $T>\delta(d_1/d_3)^{1/\gamma_3}$ and 
	$$\lambda > \left(\frac{\ln \left(Md_2(d_0\lVert C \rVert+1)\right) + \omega_+ T}{d_3(\delta T)^{\gamma_3}-d_1}\right)^{\frac{1}{\gamma_1}}>0.$$
	Then again $\alpha\in (0,1)$.
	\item\label{case3} Assume $\gamma_1 < \gamma_2$. For given $T>0$ let $\lambda >0$ large enough such that
	$$\ln \left(Md_2(d_0\lVert C \rVert+1)\right) + \omega_+ T< d_3\lambda^{\gamma_2}(\delta T)^{\gamma_3}-d_1\lambda^{\gamma_1}.$$
	Then $\alpha \in (0,1)$.
	\item Assume $\gamma_1 < \gamma_2$. Let $\lambda^*>0$ and suppose there exists $P\in\cL(X)$ such that $P_\lambda=P$ for all $\lambda >\lambda^*$, and such that the inequalities \eqref{eq:ass:uncertainty} and \eqref{eq:ass:dissipation} hold with $C_1,C_2$ as in \eqref{eq:choiceC1C2}. Then by \cite[Theorem 2.1]{GallaunST-20}, the weak observability inequality \eqref{eq:obs} holds with $\alpha=0$. 
	\item Assume $\omega_+ = 0$. Then for arbitrary $\lambda,\gamma_1,\gamma_2,\gamma_3 >0$ we can achieve $\alpha \in (0,1)$ by choosing $T>0$ large enough.
\end{enumerate}
Note that, in contrast to the cases \eqref{case1} and \eqref{case2}, in \eqref{case3} we can ensure $\alpha \in (0,1)$ for every $T>0$ by choosing $\lambda >0$ appropriately. The cases \eqref{case1}-\eqref{case3} are very similar to what was shown in \cite[Lemma 2.2]{HuangWW-20}, where the inequalities \eqref{eq:ass:uncertainty} and \eqref{eq:ass:dissipation} with \eqref{eq:choiceC1C2} where assumed to hold for all $\lambda >1$. Note that here the assumptions are only needed for some particular $\lambda >0$.
\end{remark}

By restricting to $\gamma_3 = 1$, Proposition \ref{lem:spectral+diss-obs} and the duality in Theorem \ref{thm:duality} yield the following plain sufficient condition for cost-uniform open-loop stabilizability similar to the Hilbert space result in \cite[Theorem 4.1]{LiuWXY-20}.

\begin{corollary}\label{cor:spectral+diss-control}
Let $X$ and $U$ be Banach spaces, $B\in \cL(U,X)$ and $P\in \cL(X)$ such that
\begin{align} 
\ran(P)\subset \overline{\ran(PB)}.
\label{eq:uncertainty_dual}
\end{align}
Further let $(S_t)_{t\geq 0}$ a $C_0$-semigroup on $X$, and $M \geq 1$, $\omega \in \R$ such that $\lVert S_t \rVert \leq M \euler^{\omega t}$ for all $t \geq 0$.
Assume there exist $M_P\geq 1$ and $\omega_P > \omega_+:=\max\{\omega,0\}$ such that
\begin{align}
\forall x\in X~ \forall t>0: \quad\lVert S_t (\id-P) x \rVert_{X} \le M_P \euler^{-\omega_P t}\lVert x \rVert_{X} \label{eq:dissipation_dual}.
\end{align}
Then the system \eqref{eq:ControlSystem} is cost-uniformly open-loop stabilizable. 
\end{corollary}

\begin{proof}
We apply Proposition \ref{lem:spectral+diss-obs} to the dual semigroup $(S_t')_{t\geq 0}$ on $X'$, $Y:=U'$, $C:=B'$, and $P$ replaced by its dual operator $P'$. Note that $(S_t')_{t\geq 0}$ is exponentially bounded since $(S_t)_{t\geq0}$ is exponentially bounded. The measurability of the functions $t\mapsto \norm{B'S_t'x'}_{U'}$ for all $x'\in X'$ follows from duality and the description of dual norms via the Hahn--Banach theorem. 
It is well-known, see \cite{Carja-89}, that \eqref{eq:uncertainty_dual} implies the existence of $C>0$ such that 
\begin{align*} 
\forall x'\in X': \quad\lVert P'x' \rVert_{X'} \le C \lVert B'P'x' \rVert_{U'}.
\end{align*}
Further \eqref{eq:dissipation_dual} implies 
\begin{align*}
\forall x'\in X'~ \forall t>0: \quad\lVert (\id-P') S'_t x' \rVert_{X'} \le M_P \euler^{-\omega_P t} \lVert x' \rVert_{X'}.
\end{align*}
Thus, by Proposition \ref{lem:spectral+diss-obs} with $C_1(t)=C$, $C_2(t)=M_P \euler^{-\omega_P t}$ and $\delta = (\omega_P + \omega_+)/2\omega_P$ we obtain for all $T>0$ and $r\in [1,\infty]$ that
\begin{equation*}
\forall x'\in X': \quad \norm{S_T'x'}_{X'}\leq \begin{cases}
    C_{\mathrm{obs}} \left(\int_0^T \norm{B'S_t' x'}_{U'}^{r'} \drm t\right)^{1/{r'}} + \alpha \lVert x' \rVert_{X'}& \text{if } r'\in [1,\infty),\\
    C_{\mathrm{obs}}\sup_{t\in [0,T]} \norm{C'S_t' x'}_{U'} +  \alpha \lVert x' \rVert_{X'}& \text{if } r'=\infty,
  \end{cases}
\end{equation*}
with 
\begin{align*} 
C_{\mathrm{obs}} \le \frac{2M\euler^{\omega_+ T}}{(1-\frac{\omega_+}{\omega_P})T^{1/r}}C \quad \text{and} \quad
\alpha \le M M_P\left(C\lVert B \rVert_{\cL(U,X)}+1\right) \euler^{-\frac{1}{2}(\omega_P - \omega_+)T}.
\end{align*}
For 
$$T>\frac{2\ln\left((M M_P\left(C\lVert B \rVert_{\cL(U,X)}+1\right)\right)}{\omega_P - \omega_+}$$
we have $\alpha\in [0,1)$ and the assertion follows from Theorem \ref{thm:duality} and Proposition \ref{prop:solution-stable}.
\end{proof}

\begin{remark}
The condition $\ran(P)\subset \overline{\ran(PB)}$ for the control operator $B$ does not require any constants. In applications this means that for the corresponding uncertainty principle for the dual system we do not need any assumption on the growth order of the constants in terms of the spectral parameter. An instance of this is when one considers the system \eqref{eq:ControlSystem} with $H$ being the harmonic oscillator in $L_2(\R^d)$, i.e. $H= -\Delta + \vert x \vert^2$,
 and $B$ the characteristic function of a measurable subset of $\R^d$ with positive measure. 
Indeed, it was shown in \cite[Theorem 2.1]{BeauchardJPS-21} and in \cite[Lemma 3.2]{HuangWW-20} that a spectral inequality with $P$ being any element of the spectral family associated to $H$ 
is valid under different geometric assumptions on the measurable subset with different growth orders of the constant with respect to the spectral parameter, while the dissipation estimate satisfies an estimate 
like the one in the corollary above (see, e.g., 
\cite[Eq. (4.17)]{HuangWW-20}).
\end{remark}

\begin{remark}
System \eqref{eq:ControlSystem} is called \emph{complete (or rapidly) open-loop stabilizable} if  for all $\nu>0$ the system 
\begin{equation}
\dot{x}(t) = -(A+\nu)x(t) + Bu(t), \qquad t>0, \qquad x(0) = x_0
\label{eq:ControlSystem_nu}
\end{equation}
is open-loop stabilizable. Analogously to \cite[Theorem 4.1]{LiuWXY-20}, by Corollary \ref{cor:spectral+diss-control} we obtain the following sufficient conditions for complete open-loop stabilizability: Let $(P_k)_{k\in \N}$ in $\cL(X)$ satisfying \eqref{eq:uncertainty_dual} for all $k\in\N$ and $(M_k)_{k\in \N}$ in $[1,\infty)$, $(\omega_k)_{k\in \N}$ in $\R$ with $\omega_k \to \infty$ as $k\to \infty$ such that
\begin{align*}
\forall x\in X~ \forall t>0: \quad\lVert S_t (\id-P_k) x \rVert_{X} \le M_k \euler^{-\omega_k t}\lVert x \rVert_{X}.
\end{align*}
Then \eqref{eq:ControlSystem} is complete open-loop stabilizable. Indeed, for all $\nu>0$ there exists $k\in \N$ such that $\omega_k > \omega_+ + \nu$ and by Corollary \ref{cor:spectral+diss-control} the system \eqref{eq:ControlSystem_nu} is open-loop stabilizable.
\end{remark}

\section{Application: Fourier Multipliers and Fractional Powers}\label{sec:application}

We denote by $\mathcal{S}(\R^d)$ the Schwartz space of rapidly decreasing functions, which is dense in $L_p(\R^d)$ for all $p \in[1,\infty)$. The space of tempered distributions, i.e.\ the topological dual space of $\mathcal{S}(\R^d)$, is denoted by $\mathcal{S}'(\R^d)$. We define the Fourier transformation $\F\from \mathcal{S}(\R^d)\to\mathcal{S}(\R^d)$ by
\[\F f (\xi) := \int_{\R^d} f(x) \euler^{-\ii\xi\cdot x}\drm x \quad(\xi\in\R^d).\]
Then $\F$ is bijective, continuous, and has a continuous inverse given by
\[\F^{-1} f(x) = \frac{1}{(2\pi)^d} \int_{\R^d} f(\xi) \euler^{\ii x\cdot \xi}\drm \xi \quad(x\in\R^d)\]
for all $f\in \mathcal{S}(\R^d)$. 
By duality, we can extend the Fourier transformation as a bijection on $\mathcal{S}'(\R^d)$ as well.

Let $m \in \N$ and $a\from\R^d\to\C$,
\begin{equation*}
 a(\xi) := \sum _{\abs{\alpha} \leq m} a_\alpha \xi^\alpha \quad (\xi \in \R^d),
\end{equation*}
be a polynomial of degree $m$ with coefficients $a_\alpha \in \C$ and assume that $a$ is strongly elliptic, i.e.\ there exists $c>0$ and $\omega\in\R$ such that
\[\Re a(\xi) \geq c\abs{\xi}^m - \omega \quad(\xi\in \R^d).\]
Let $s\in (0,1]$. Then
\[\Re((a(\xi)+\omega)^s) \geq (\Re a(\xi)+\omega)^s \geq c^s \abs{\xi}^{sm} \quad(\xi\in\R^d).\]

Let $\tilde{m}\in\N_0$ be the largest integer less than $sm$, and $b\from\R^d\to\C$,
\[b(\xi):=\sum _{\abs{\alpha} \leq \tilde{m} } b_\alpha \xi^\alpha \quad (\xi \in \R^d).\]
We consider $a_{s,b}:= (a+\omega)^s+b$. 
Then there exists $\nu\in\R$ such that
\begin{equation}
\label{eq:realpart}
\Re a_{s,b}(\xi) = \Re (a(\xi)+\omega)^s + \Re b(\xi) \geq c^s\abs{\xi}^{sm} - \nu \quad(\xi\in\R^d).
\end{equation}

Note that $a_{s,b}$ may not be differentiable at $0$. However, it can be shown that for $t>0$ we have $e^{-ta_{s,b}}\in L_1(\R^d)$ and $\F^{-1}e^{-ta_{s,b}}\in L_1(\R^d)$. Indeed, $e^{-ta_{s,b}}$ decays faster than any polynomial. Thus, $e^{-ta_{s,b}}\in L_1(\R^d)$ and $\F^{-1}e^{-ta_{s,b}}\in C^\infty(\R^d)$. Moreover, the Riemann--Lebesgue lemma yields $\F^{-1}e^{-ta_{s,b}}\in C_0(\R^d)$. Then by subordination techniques (see e.g.\ \cite{KruseMS-21}), one can show that $f\mapsto \F^{-1}e^{-ta_{s,0}}* f$ yields a bounded operator on $L_1(\R^d)$. By a perturbation argument, also $f\mapsto \F^{-1}e^{-ta_{s,b}}* f$ is bounded on $L_1(\R^d)$. Since this operator is also translation invariant, $\F^{-1}e^{-ta_{s,b}}$ is given by a finite Borel measure (cf.\ \cite[Theorem 2.58]{Grafakos-08}) and therefore $\F^{-1}e^{-ta_{s,b}}\in L_1(\R^d)$.
%
%
%
%
%

Taking into account Young's inequality, for $p\in[1,\infty]$ and $t\geq 0$ we define $S^{(s),p}_t\from L_p(\R^d)\to L_p(\R^d)$ by
\begin{equation*} 
 S^{(s),p}_0f:= f,\quad S^{(s),p}_tf := \F^{-1}\euler^{-ta_{s,b}}*f \quad (t>0).
\end{equation*}
It is easy to see that $S^{(s),p}$ is a $C_0$-semigroup for $p\in [1,\infty)$ and $S^{(s),\infty}$ is a weak$^*$ continuous exponentially bounded semigroup.

\begin{definition}
A set $E\subset \R^d$ is called \emph{thick} if $E$ is measurable and there exist $\rho\in (0,1]$ and $L\in (0,\infty)^d$ such that
\[\abs{E\cap \Bigl(\bigtimes_{i=1}^d (0,L_i) + x\Bigr)} \geq \rho \prod_{i=1}^d L_i \quad(x\in\R^d).\]
\end{definition}

\begin{proposition}[{Logvinenko--Sereda theorem, see e.g.\ \cite{Kovrijkine-01}}]
\label{prop:Logvinenko-Sereda}
    Let $E\subset \R^d$ be thick. Then there exist $d_0,d_1>0$ such that for all $p\in [1,\infty]$, all $\lambda>0$ and all $f\in L_p(\R^d)$ with $\supp \F f \subset  [-\lambda,\lambda]^d$ we have
    \[\norm{f}_{L_p(\R^d)} \leq d_0 \euler^{d_1\lambda}\norm{f}_{L_p(E)} \quad(f\in L_p(\R^d)).\]
\end{proposition}

Let $\eta\in C_{\mathrm c}^\infty ([0,\infty) )$ with $0\leq\eta\leq 1$ such that $\eta (r) = 1$ for $r\in [0,1/2]$ and $\eta (r) = 0$ for $r\geq 1$. 
For $\lambda > 0$ we define $\chi_\lambda\from \R^d\to \R$ by $\chi_\lambda (\xi) = \eta (\lvert \xi \rvert / \lambda)$. Since $\chi_\lambda \in \mathcal{S}(\R^d)$, we have $\mathcal{F}^{-1}\chi_\lambda \in \mathcal{S}(\R^d)$ and for all $p \in [1,\infty]$ we define $P_\lambda \from  L_p(\R^d) \to L_p(\R^d)$ by $P_\lambda f = (\mathcal{F}^{-1}\chi_\lambda) * f$.

\begin{proposition}
\label{prop:Dissipation}
    There exists $K\geq 0$ such that for all $s\in (0,1]$, $p\in[1,\infty]$ and all $\lambda>(2^{sm+4}\nu_+/c^s)^{1/(sm)}$, $t\geq 0$ and $f\in L_p(\R^d)$ we have
    \[\norm{(I-P_\lambda) S^{(s),p}_t f}_p \leq K \euler^{-2^{-sm-4}c^st\lambda^{sm}}\norm{f}_p.\]
\end{proposition}

\begin{proof}
 (i) We first show the corresponding estimate for $a_{s,b}(\xi)=\abs{\xi}^{sm}$.
 
 The proof is an adaptation of the proof of \cite[Proposition 3.2]{BombachGST-20}, so we only sketch the details.
 Let $f\in L_p(\R^d)$. Then
 \[(I-P_\lambda)S^{(s),p}_t f = \F^{-1}\bigl((1-\chi_\lambda) \euler^{-ta_{s,b}}\bigr)*f.\]
 With $k_\mu:=\F^{-1}\bigl((1-\chi_\mu) \euler^{-a_{s,b}}\bigr)$ we observe
 \[\norm{\F^{-1}\bigl((1-\chi_\lambda) \euler^{-ta_{s,b}}\bigr)}_{L_1(\R^d)} = \norm{k_{t^{1/(sm)}\lambda}}_{L_1(\R^d)},\]
 so by Young's inequality it suffices to estimate $\norm{k_\mu}_{L_1(\R^d)}$.
 Using that the inverse Fourier transform maps differentiation to multiplication, for $\alpha\in\N_0^d$ we observe
 \[\abs{x^\alpha k_\mu(x)} \leq \frac{1}{(2\pi)^d} \int_{\R^d} \abs{\partial_\xi ^\alpha \bigl((1-\chi_\mu(\xi)) \euler^{-\abs{\xi}^{sm}}\bigr)}\,\drm \xi \quad(x\in\R^d).\]
 Estimating the derivatives in the integrand for $\abs{\alpha}\leq d+1$, we find $K_1\geq 0$ such that
 \[\abs{x^\alpha k_\mu(x)} \leq K_1 \euler^{-\mu^{sm}/(2^{sm+2})}\quad(x\in\R^d).\]
 Thus, there exists $K\geq 0$ such that
 \[\norm{k_\mu}_{L_1(\R^d)} \leq K \euler^{-\mu^{sm}/(2^{sm+2})}\]
 and therefore
 \[\norm{(I-P_\lambda) S^{(s),p}_t f}_p \leq K \euler^{-2^{-sm-2}t\lambda^{sm}}\norm{f}_p.\]
 
 (ii)
 For the general case, we follow the perturbation argument in \cite[Proposition 3.3]{BombachGST-20}. Let $\tilde{a}(\xi):=\tfrac{c^s}{2} \abs{\xi}^{sm}$ and denote the corresponding semigroup by $\widetilde{S}$. Then by (i) we have
 \[\norm{(I-P_\lambda) \widetilde{S}_t f}_p \leq K \euler^{-2^{-sm-3}tc^s\lambda^{sm}}\norm{f}_p.\]
 Moreover, $a_{s,b} = (a_{s,b}-\tilde{a}) + \tilde{a}$ and $a_{s,b}-\tilde{a}$ satisfies an estimate similar to \eqref{eq:realpart}, so the corresponding semigroup $(T_t)_{t\geq 0}$ obeys an exponential bound of the form 
 \[\norm{T_t} \leq M \euler^{\nu t} \quad(t\geq 0).\]
 Thus, since $S^{(s),p}_t = T_t\widetilde{S}_t$ and Fourier multipliers commute, we arrive at
 \begin{align*}
    \norm{(I-P_\lambda) S^{(s),p}_t f}_p & = \norm{S^{(s),p}_t(I-P_\lambda) f}_p
    \leq \norm{T_t} \norm{\widetilde{S}_t(I-P_\lambda) f}_p\\
    & \leq M K \euler^{-t(2^{-sm-3}c^s\lambda^{sm}-\nu)}\norm{f}_p.
\end{align*}
Now, for $\lambda>(2^{sm+4}\nu_+/c^s)^{1/(sm)}$ we have
$2^{-sm-3}c^s\lambda^{sm}-\nu > 2^{-sm-4} c^s \lambda^{sm}$.
\end{proof}
 
In view of Proposition \ref{prop:Logvinenko-Sereda} and Proposition \ref{prop:Dissipation}, we can apply Proposition \ref{lem:spectral+diss-obs} and obtain various weak observability estimates by the cases in Remark \ref{rem:application} with $\gamma_1=1$, $\gamma_2=sm$ and $\gamma_3=1$. We state this as a corollary.

\begin{corollary}
\label{cor:weak_obs}
    Let $p\in[1,\infty]$, $s\in (0,1]$.
    \begin{enumerate}[(a)]
        \item\label{parta} Let $s \leq 1 /m $. Then there exists $T>0$ such that the semigroup $(S_t^{(s),p})_{t\geq 0}$ satisfies a weak observability inequality with some $\alpha\in (0,1)$.
        \item Let $s> 1 / m$. Then for all $T>0$ the semigroup $(S_t^{(s),p})_{t\geq 0}$ satisfies a weak observability inequality with $\alpha=0$.
    \end{enumerate}
\end{corollary}

In view of Theorem \ref{thm:duality}, by duality we thus obtain statements on cost-uniform $\alpha$-con\-trol\-labi\-li\-ty and approximate null-controllability, and in view of Proposition \ref{prop:solution-stable} also for cost-uniform open-loop stabilizability. Note that for the fractional Laplacian $-A=-(-\Delta)^s$ in $L_2(\R^d)$, the system is not approximately null-controllable for $s< 1 /2$, cf.\ \cite{HuangWW-20,Koenig-20}.
For Corollary \ref{cor:weak_obs}\eqref{parta} even more is true. By invoking that we prove the uncertainty principle and the dissipation estimate for all $\lambda > \lambda_0$ with some $\lambda_0\geq 0$, we get, by using Remark \ref{rem:application}\ref{case1} for $T>0$ large enough, that for all $\alpha\in (0,1)$ there is $T>0$ such that $(S_t^{(s),p})_{t\geq 0}$ satisfies a weak observability inequality.


\begin{thebibliography}{LWXY20}

\bibitem[AM21]{AlphonseM-21}
P.~Alphonse and J.~Martin.
\newblock Stabilization and approximate null-controllability for a large class
  of diffusive equations from thick control supports.
\newblock arXiv:2101.03772 [math.AP], 2021.

\bibitem[BGST21]{BombachGST-20}
C.~Bombach, D.~Gallaun, C.~Seifert, and M.~Tautenhahn.
\newblock Observability and null-controllability for parabolic equations in
  $l_p$-spaces.
\newblock arXiv:2005.14503 [math.FA], 2021.

\bibitem[BJP21]{BeauchardJPS-21}
K.~Beauchard, P.~Jaming, and K.~{Pravda-Starov}.
\newblock Spectral estimates for finite combinations of {H}ermite functions and
  null-controllability of hypoelliptic quadratic equations.
\newblock {\em Studia Math.}, 260:1--43, 2021.

\bibitem[BPS18]{BeauchardP-18}
K.~Beauchard and K.~Pravda-Starov.
\newblock Null-controllability of hypoelliptic qua\-dra\-tic differential
  equations.
\newblock {\em J. \'Ec. polytech. Math.}, 5:1--43, 2018.

\bibitem[C{\^a}r89]{Carja-89}
O.~C{\^a}rj{\u{a}}.
\newblock Range inclusion for convex processes on {B}anach spaces; applications
  in controllability.
\newblock {\em Proc. Amer. Math. Soc.}, 105(1):185--191, 1989.

\bibitem[Gol66]{Goldberg-66}
S.~Goldberg.
\newblock {\em Unbounded linear operators: Theory and applications}.
\newblock McGraw-Hill Book Company, New York, 1966.

\bibitem[Gra08]{Grafakos-08}
L.~Grafakos.
\newblock {\em Classical {F}ourier analysis}.
\newblock Springer, New York, 2008.

\bibitem[GST20]{GallaunST-20}
D.~Gallaun, C.~Seifert, and M.~Tautenhahn.
\newblock Sufficient criteria and sharp geometric conditions for observability
  in {B}anach spaces.
\newblock {\em SIAM J. Control Optim.}, 58(4):2639--2657, 2020.

\bibitem[HWW21]{HuangWW-20}
S.~Huang, G.~Wang, and M.~Wang.
\newblock Characterizations of stabilizable sets for some parabolic equations
  in $\mathbb{R}^n$.
\newblock {\em J. Differential Equations}, 272:255--288, 2021.

\bibitem[JL99]{JerisonL-99}
D.~Jerison and G.~Lebeau.
\newblock Nodal sets of sums of eigenfunctions.
\newblock In M.~Christ, {C.~E.} Kenig, and C.~Sadosky, editors, {\em Harmonic
  Analysis and Partial Differential Equations}, Chicago Lectures in
  Mathematics, pages 223--239. University of Chicago Press, Chicago, IL, 1999.

\bibitem[KMS21]{KruseMS-21}
K.~Kruse, J.~Meichsner, and C.~Seifert.
\newblock Subordination for sequentially equicontinuous equibounded $ c_0
  $-semigroups.
\newblock {\em J. Evol. Equ.}, 21(2):2665--2690, 2021.

\bibitem[Koe20]{Koenig-20}
A.~Koenig.
\newblock Lack of null-controllability for the fractional heat equation and
  related equations.
\newblock {\em SIAM J. Control Optim.}, 58(6):3130--3160, 2020.

\bibitem[Kov01]{Kovrijkine-01}
O.~Kovrijkine.
\newblock Some results related to the {L}ogvinenko-{S}ereda {T}heorem.
\newblock {\em Proc. Amer. Math. Soc.}, 129(10):3037--3047, 2001.

\bibitem[Lis20]{Lissy-20}
P.~Lissy.
\newblock A non-controllability result for the half-heat equation on the whole
  line based on the prolate spheroidal wave functions and its application to
  the {G}rushin equation.
\newblock hal-02420212, 2020.

\bibitem[LR95]{LebeauR-95}
G.~Lebeau and L.~Robbiano.
\newblock Contr{\^o}le exact de l'{\'e}quation de la chaleur.
\newblock {\em Comm. Partial Differential Equations}, 20(1--2):335--356, 1995.

\bibitem[LSS20]{LenzSS-20}
{H. D.} Lenz, P.~Stollmann, and G.~Stolz.
\newblock An uncertainty principle and lower bounds for the {D}irichlet
  {L}aplacian on graphs.
\newblock {\em J. Spectr. Theory}, 10(1):115--145, 2020.

\bibitem[LWXY20]{LiuWXY-20}
H.~Liu, G.~Wang, Y.~Xu, and H.~Yu.
\newblock Characterizations on complete stabilizability.
\newblock arXiv:2012.07253 [math.OC], 2020.

\bibitem[LZ98]{LebeauZ-98}
G.~Lebeau and E.~Zuazua.
\newblock Null-controllability of a system of linear thermoelasticity.
\newblock {\em Arch. Ration. Mech. Anal.}, 141(4):297--329, 1998.

\bibitem[Mil10]{Miller-10}
L.~Miller.
\newblock A direct {L}ebeau-{R}obbiano strategy for the observability of
  heat-like semigroups.
\newblock {\em Discrete Contin. Dyn. Syst. Ser. B}, 14(4):1465--1485, 2010.

\bibitem[NTTV20]{NakicTTV-20}
I.~Naki\'c, M.~T\"aufer, M.~Tautenhahn, and I.~Veseli\'c.
\newblock Sharp estimates and homogenization of the control cost of the heat
  equation on large domains.
\newblock {\em ESAIM Control Optim. Calc. Var.}, 26(54):26 pages, 2020.

\bibitem[TT11]{TenenbaumT-11}
G.~Tenenbaum and M.~Tucsnak.
\newblock On the null-controllability of diffusion equations.
\newblock {\em ESAIM Control Optim. Calc. Var.}, 17(4):1088--1100, 2011.

\bibitem[TWX20]{TrelatWX-20}
E.~Tr{\'e}lat, G.~Wang, and Y.~Xu.
\newblock Characterization by observability inequalities of controllability and
  stabilization properties.
\newblock {\em Pure Appl. Anal.}, 2(1):93--122, 2020.

\bibitem[Vie05]{Vieru-05}
A.~Vieru.
\newblock On null controllability of linear systems in {B}anach spaces.
\newblock {\em Systems Control Lett.}, 54(4):331--337, 2005.

\bibitem[WZ17]{WangZ-17}
G.~Wang and C.~Zhang.
\newblock Observability inequalities from measurable sets for some abstract
  evolution equations.
\newblock {\em SIAM J. Control Optim.}, 55(3):1862--1886, 2017.

\bibitem[Zab08]{Zabczyk-08}
J.~Zabczyk.
\newblock {\em Mathematical control theory: An introduction}.
\newblock Birkh\"auser, Boston, 2008.

\end{thebibliography}

\end{document}